\documentclass[11pt,fleqn]{amsart}
\usepackage[utf8]{inputenx}
\usepackage[T1]{fontenc}
\usepackage{newunicodechar}

\newunicodechar{α}{\alpha}
\newunicodechar{β}{\beta}
\newunicodechar{γ}{\gamma}
\newunicodechar{δ}{\delta}
\newunicodechar{ε}{\epsilon}
\newunicodechar{ζ}{\zeta}
\newunicodechar{η}{\eta}
\newunicodechar{θ}{\theta}
\newunicodechar{ι}{\iota}
\newunicodechar{κ}{\kappa}
\newunicodechar{λ}{\lambda}
\newunicodechar{μ}{\mu}
\newunicodechar{ν}{\nu}
\newunicodechar{ξ}{\xi}
\newunicodechar{ο}{\omicron}
\newunicodechar{π}{\pi}
\newunicodechar{ρ}{\rho}
\newunicodechar{σ}{\sigma}
\newunicodechar{τ}{\tau}
\newunicodechar{υ}{\upsilon}
\newunicodechar{φ}{\phi}
\newunicodechar{χ}{\chi}
\newunicodechar{ψ}{\psi}
\newunicodechar{ω}{\omega}

\newunicodechar{Α}{\Alpha}
\newunicodechar{Β}{\Beta}
\newunicodechar{Γ}{\Gamma}
\newunicodechar{Δ}{\Delta}
\newunicodechar{Ε}{\Epsilon}
\newunicodechar{Ζ}{\Zeta}
\newunicodechar{Η}{\Eta}
\newunicodechar{Θ}{\Theta}
\newunicodechar{Ι}{\Iota}
\newunicodechar{Κ}{\Kappa}
\newunicodechar{Λ}{\Lambda}
\newunicodechar{Μ}{\Mu}
\newunicodechar{Ν}{\Nu}
\newunicodechar{Ξ}{\Xi}
\newunicodechar{Ο}{\Omicron}
\newunicodechar{Π}{\prod}
\newunicodechar{Ρ}{\Rho}
\newunicodechar{Σ}{\sum}
\newunicodechar{Τ}{\Tau}
\newunicodechar{Υ}{\Upsilon}
\newunicodechar{Φ}{\Phi}
\newunicodechar{Χ}{\Chi}
\newunicodechar{Ψ}{\Psi}
\newunicodechar{Ω}{\Omega}

\newunicodechar{𝒳}{\mathcal{X}}
\newunicodechar{𝒴}{\mathcal{Y}}
\newunicodechar{ℳ}{\mathcal{M}}

\newunicodechar{ℕ}{\mathbb{N}}
\newunicodechar{ℤ}{\mathbb{Z}}
\newunicodechar{ℚ}{\mathbb{Q}}
\newunicodechar{ℝ}{\mathbb{R}}
\newunicodechar{ℂ}{\mathbb{C}}
\newunicodechar{ℍ}{\mathbb{H}}
\newunicodechar{ℙ}{\mathbb{P}}

\newunicodechar{∅}{\emptyset}
\newunicodechar{≠}{\neq}
\newunicodechar{≅}{\cong}
\newunicodechar{∂}{\partial}
\newunicodechar{∊}{\in}
\newunicodechar{∉}{\notin}
\newunicodechar{∀}{\forall}
\newunicodechar{∃}{\exists}
\newunicodechar{∞}{\infty}
\newunicodechar{→}{\rightarrow}
\newunicodechar{⇒}{\implies}
\newunicodechar{⇔}{\iff}
\newunicodechar{≤}{\leq}
\newunicodechar{≥}{\geq}
\newunicodechar{⊂}{\subset}
\newunicodechar{⊆}{\subseteq}
\newunicodechar{⊃}{\supset}
\newunicodechar{⊇}{\supseteq}
\newunicodechar{⋂}{\cap}
\newunicodechar{⋃}{\cup}
\newunicodechar{∫}{\int}
\newunicodechar{⋀}{\wedge}
\newunicodechar{⊗}{\otimes}
\newunicodechar{⊕}{\oplus}

\usepackage{amssymb}
\usepackage{amsmath}
\usepackage{mathtools}
\usepackage{enumitem}
\usepackage{mathrsfs}
\usepackage[all]{xy}
\usepackage{mathtools}
\usepackage{hyperref}
\usepackage{cleveref}
\usepackage{url}
\usepackage{comment}
\usepackage{lipsum} 
\usepackage{tikz-cd}
\usepackage{tikz}
\usepackage{capt-of}

\usepackage[british]{babel}
\usepackage[margin=1.2in]{geometry}
\newcommand{\SL}{\textup{SL}}

\def\lf{\left\lfloor}   
\def\rf{\right\rfloor}
\newtheorem{lemma}{Lemma}[section]

\newtheorem{theorem}[lemma]{Theorem}
\newtheorem{proposition}[lemma]{Proposition}

\theoremstyle{definition}

\newtheorem{definition}[lemma]{Definition}
\newtheorem{remark}[lemma]{Remark}

\newcommand\R{{\mathbb R}}
\newcommand\K{{\mathbb K}}

\newcommand\Z{\mathbb{Z}}

\newcommand{\Q}{\mathbb{Q}}

\newcommand\C{\mathbb{C}}

\newcommand{\HH}{{\mathbb H}}

\newcommand{\OO}{\mathcal{O}}

\newtheoremstyle{break}
  {\topsep}{\topsep}%
  {\itshape}{}%
  {\bfseries}{}%
  {\newline}{}%
\theoremstyle{break}
\newtheorem{theoremb}[lemma]{Theorem}
\theoremstyle{plain}                    

\theoremstyle{definition}
\newcommand{\drawgeo}[4]{%
  \pgfmathsetmacro{\xone}{#1}%
  \pgfmathsetmacro{\yone}{#2}%
  \pgfmathsetmacro{\xtwo}{#3}%
  \pgfmathsetmacro{\ytwo}{#4}%
  \pgfmathtruncatemacro{\isvert}{ifthenelse(abs(\xone-\xtwo)<1e-8,1,0)}%
  \ifnum\isvert=1
    \draw[line width=1pt] (\xone,\yone) -- (\xtwo,\ytwo);
  \else
    \pgfmathsetmacro{\xc}{(\xtwo*\xtwo+\ytwo*\ytwo-\xone*\xone-\yone*\yone)/(2*(\xtwo-\xone))}%
    \pgfmathsetmacro{\r}{sqrt((\xone-\xc)^2 + (\yone)^2)}%
    \pgfmathparse{atan2(\yone,\xone-\xc)} \let\angone\pgfmathresult
    \pgfmathparse{atan2(\ytwo,\xtwo-\xc)} \let\angtwo\pgfmathresult
    \pgfmathsetmacro{\delta}{\angtwo-\angone}%
    \pgfmathparse{(\delta>180)?\angtwo-360:((\delta<-180)?\angtwo+360:\angtwo)} \let\angtwoadj\pgfmathresult
    \draw[line width=1pt] (\xc,0) ++(\angone:\r) arc (\angone:\angtwoadj:\r);
  \fi
}    

\usepackage{tikz}

\def\hypquad(#1,#2)(#3,#4)(#5,#6)(#7,#8)#9{
    \pgfmathsetmacro{\xcA}{(#3*#3 - #1*#1 + #4*#4 - #2*#2)/(2*(#3 - #1))}
    \pgfmathsetmacro{\rA}{sqrt((#1-\xcA)^2 + #2*#2)}
    \pgfmathsetmacro{\startA}{atan2(#2, #1-\xcA)}
    \pgfmathsetmacro{\endA}{atan2(#4, #3-\xcA)}

    \pgfmathsetmacro{\xcB}{(#5*#5 - #3*#3 + #6*#6 - #4*#4)/(2*(#5 - #3))}
    \pgfmathsetmacro{\rB}{sqrt((#3-\xcB)^2 + #4*#4)}
    \pgfmathsetmacro{\startB}{atan2(#4, #3-\xcB)}
    \pgfmathsetmacro{\endB}{atan2(#6, #5-\xcB)}

    \pgfmathsetmacro{\xcC}{(#7*#7 - #5*#5 + #8*#8 - #6*#6)/(2*(#7 - #5))}
    \pgfmathsetmacro{\rC}{sqrt((#5-\xcC)^2 + #6*#6)}
    \pgfmathsetmacro{\startC}{atan2(#6, #5-\xcC)}
    \pgfmathsetmacro{\endC}{atan2(#8, #7-\xcC)}

    \pgfmathsetmacro{\xcD}{(#1*#1 - #7*#7 + #2*#2 - #8*#8)/(2*(#1 - #7))}
    \pgfmathsetmacro{\rD}{sqrt((#7-\xcD)^2 + #8*#8)}
    \pgfmathsetmacro{\startD}{atan2(#8, #7-\xcD)}
    \pgfmathsetmacro{\endD}{atan2(#2, #1-\xcD)}

    \draw[fill=#9, draw=black] (#1, #2)
        arc[radius=\rA, start angle=\startA, end angle=\endA]
        arc[radius=\rB, start angle=\startB, end angle=\endB]
        arc[radius=\rC, start angle=\startC, end angle=\endC]
        arc[radius=\rD, start angle=\startD, end angle=\endD] -- cycle;
}

\graphicspath{{./}{images/}}

\makeatletter
\def\namedlabel#1#2{\begingroup
   \def\@currentlabel{#2}
   \label{#1}\endgroup
}
\makeatother

\makeindex

\begin{document}

\title[Shimura curves of discriminants 14 and 15 and associated Heun functions]{Shimura curves of discriminant 14 and 15 and associated Heun Functions} 
\begin{abstract}
    We study the Shimura curves $X_D^*(1)$ of discriminant $D$ for $D=14, 15$ relating specific weight $4$  modular forms on these Shimura curves to the associated Heun functions for the quadrilateral group. We also discuss how the Picard-Fuchs equation of the associated family of abelian surfaces has solutions coming from modular forms on $X^*_{D}(1)$. We determine the rational exceptional sets of the associated Heun functions, and the algebraic values attained by the Heun function at these  points. For example,

$${\rm He}\left( 81, \frac{1}{2}; \frac{1}{3}, \frac{1}{6}, \frac{1}{2}, \frac{1}{2};-\frac{729}{112}\right)= \left( \frac{2^2 \cdot 3^3 \cdot 5^3}{7^5} \right)^{\frac{1}{6}} $$.

\end{abstract}
\author[Harshavardhan Reddy]{Harshavardhan Reddy}
\email{harsha.v.r@gmail.com} 

\author[Devendra Tiwari]{Devendra Tiwari}
\address{Statistics and Mathematics Department\\
Indian Statistical Institute\\
Bangalore 560 059\\
INDIA}
\email{devendra9.dev@gmail.com}

%\subjclass[2000]{Primary 20H10; Secondary 51M10}
\date{\today}

\thispagestyle{empty}

\maketitle

\section{Introduction}
The set $E(f)$ for a complex function $f(z)$ is defined as the set of those algebraic numbers in the domain of $f$ where the value of $f$ is also algebraic. For some transcendental functions such as $f(z)=e^z$, the size of $E(f)$ can be very small. A classical result of Hermite, Lindemann and Weierstrass states that $E(e^z)=\{0\}$. So, $E(f)$ is referred to as the exceptional set of $f$. Calculating these exceptional sets for transcendental functions is a classical problem. The problem (initiated by Siegel) of the exceptional set associated with the Gauss hypergeometric function $_2F_1(a,b,c;z)$ of one variable has been solved completely by Wolfart using Wüstholz's Analytic Subgroup Theorem, see \cite{jw}.  Wolfart showed that a hypergeometric function is algebraic for infinitely many algebraic values of $t$ if and only if the monodromy group
of the corresponding differential equation is arithmetic.\\

As an application of Wolfart's result, Beukers and Wolfart  \cite{bw} explicitly determined the  exceptional sets of hypergeometric functions 
with monodromy group isomorphic to $\SL(2, \Z)$, and derived explicit formulas for values of these functions at points in these sets. The same method was extended by Archinard in \cite{na} to compute exceptional sets of hypergeometric functions arising
from the Picard-Fuchs equations that correspond to the $(2,3, ∞)$ triangle group. Yang in \cite{yy3} used relations between hypergeometric functions and certain Borcherds forms to compute special  values of hypergeometric functions, 
where the transcendental factor comes from certain periods associated to elliptic curves with complex multiplication (CM). The exposition in \cite{cw} relates Wolfart's work with a special case of the André–Oort conjecture on the distribution of complex multiplication (or special) points on algebraic curves in Shimura varieties. Tretkoff's exposition \cite{pt} contains an extensive bibliography on transcendence of special values of modular and hypergeometric functions.
%TODO citation for Archimard, Lindemann
\begin{definition} Let $B$ be an indefinite quaternion algebra of discriminant $D$ with a maximal order $\OO$. Let $Γ=\{a∊\OO^*: N(a)=1\}/\{\pm1\}, Γ^*=\{[a]∊B^*/ℚ^* : a\OO=\OO a, N(a)>0\}$. Define the Shimura curves by $X_D(1)=ℍ/Γ, X_D^*(1)=X_D(1)/W_D$ where $W_D$ is the group generated by Atkin-Lehner maps $ω_k,\ k|D$. The constructed curve is independent of the maximal order up to isomorphism as all maximal orders are conjugate.
\\
For a positive integer $N$ with $(D,N)=1$, we can choose an Eichler order $\OO_N$ of level $N$. We can similarly define the groups $Γ(N)$ and the Shimura curves $X_D(N), X_D^*(N)=X_D(N)/W_D$. Discriminants 14 and 15 give indefinite quaternion algebras for which all the Eichler orders of level $N$ are conjugate.  
\end{definition}

\subsection{Heun Functions and Shimura Curves}  Extending Wolfart's study of the exceptional set for $_2F_1(a,b,c;z)$, Baba and Granath characterized in \cite{bg2}, for the $(2, 4, 6)$ triangle
group, the points where $_2F_1(\frac{1}{24},\frac{5}{24}, \frac{3}{4};z)$ has a fixed transcendental part. For this they related the hypergeometric function to modular forms on an associated Shimura curve $X_6^*(1)$ of discriminant 6, which is uniformized by the $(2,4,6)$ triangle group. They subsequently extended this study to Shimura curves $X_{10}^*(1)$ of discriminant 10 in \cite{bg}. This curve is a coarse moduli space of principally polarized abelian surfaces $A_z$ over $\C$ with
quaternionic multiplication by the quaternion algebra $B$ of discriminant 10 over $\Q$. They calculated the exceptional sets for the  Heun function since the associated differential equation has $4$ singular points. In \cite{bg3}, they further extend their results to calculating exceptional sets in a general situation. We use this result to get the exceptional set, see Proposition \ref{thm:exception-set}.

\subsection{Main Result} In this work, we extend past results in \cite{bg} to determine the exceptional sets and associated algebraic values to two other genus $0$ Shimura curves of discriminant 14 and 15.  By the rational exceptional set, we mean the intersection of the exceptional set with $ℚ$.

 \newtheorem*{mainthm}{Main Result}
\begin{mainthm}
Let $X_{14}^{*}(1)$ and $X_{15}^{*}(1)$ be the Shimura curves of discriminant 14 and 15 respectively. They are uniformized by the arithmetic quadrilateral Fuchsian groups of signatures $(0;2,2,2,4)$ and $(0;2,2,2,6)$, respectively. We will use Hauptmoduln $u$ and $t$ on these curves respectively.
\begin{enumerate}
    \item For $D=14$, the graded ring of quaternionic modular forms is generated by forms $g_4, g_8, g_{14}$. The associated differential equation with rational coefficients has a scaled Heun function $F(u)=He(\frac{r_2}{r_1}, \frac{-3}{256.r_1}, \frac{1}{8},\frac{3}{8},\frac{1}{2},\frac{1}{2},u/r_1)$  as a solution, where $r_1, r_2$ are the conjugate roots of $(16u^2+13u+8)=0$.\\ The form $g_4^{1/4}$ is also a  solution to this differential equation and the rational exceptional set of F is
    $$ E(F) \cap \mathbb{Q} = \{0,\frac{75}{16}\}, \quad \text{with} \quad F\left(\frac{75}{16}\right)= \frac{\sqrt{7}}{3}$$
    \item For $D=15$, the graded ring of modular forms is generated by $h_4, h_8, h_{10}, h_{12}$. The normalized form $h_4^{1/4}$ satisfies the associated Heun equation, and the rational exceptional set of its corresponding Heun function is exactly:
    $$ E \cap \mathbb{Q} = \left\{ 0, 243, -\frac{729}{112} \right\}. $$
    Furthermore, at the exceptional CM points $t = 243$ and $t = -\frac{729}{112}$, the Heun function evaluates explicitly to the following algebraic numbers:
    $$ He\left(81, \frac{1}{2}; \frac{1}{3}, \frac{1}{6}, \frac{1}{2}, \frac{1}{2}; 243\right) = \left(\frac{\sqrt{3}}{2} + \frac{i}{4}\right), $$
    $$ He\left(81, \frac{1}{2}; \frac{1}{3}, \frac{1}{6}, \frac{1}{2}, \frac{1}{2}; -\frac{729}{112}\right) = \left(\frac{2^2 \cdot 3^3 \cdot 5^3}{7^5}\right)^{1/6}. $$
\end{enumerate}
\end{mainthm}
 
 \subsection{Idea of proof} 
 The main difficulty in extending the results from D=10 to the D=14,15 cases is that a suitable factorization of the following covering map need not be available.  We would like to factorize $t:ℍ→X_{D}^*(1)$ as $ℍ\xrightarrow[]{t_N}X_D^*(N)\xrightarrow[]{η}X_D^*(1),\ t=η\circ t_N$ such that the Atkin-Lehner maps on $X_D^*(N)$ map the image of a base CM point (where $t=0$) to the image of the CM point, where we want to compute the Heun function.   This is useful as the Hauptmodul derivative $dt/dz$ which appears in  Heun function expressions, and is hard to compute explicitly, can be expressed in terms of $dt_N/dz$ and $dη/d{t_N}$. $dη/d{t_N}$ is explicitly computable and $dt_N/dz$ cancels when computing ratios. See Theorem \ref{thm:heun_factorization}. In the absence of such a factorization, we develop an alternative approach based on the Yang polynomial as defined in \cite{yy1}.\\
 \\
 \medskip{}
 We first identify the arithmetic quadrilateral Fuchsian groups uniformizing the genus-zero Shimura curves $X_{14}^*(1)$ and $X_{15}^*(1)$ in Section 2. Then we find an explicit expression for the modular form using Yang’s description \cite{yy}. Along the way, we describe the ring of modular forms, which may be of independent interest. We show that a natural differential equation associated with certain weight $4$ modular forms is the associated Heun equation for the Shimura curve. This implies that the fractional-weight modular forms, $g_4^{1/4}$ and $h_4^{1/4}$ for $X_{14}^*(1)$ and $X_{15}^*(1)$ respectively, serve as explicit local solutions in a neighbourhood around the basepoint, to the associated Heun equations. \\
 
 This helps us to calculate the Heun function in terms of this modular form in this neighbourhood. In Section 4 we determine rational exceptional sets by using a result in \cite{bg3}, establishing that the Heun function evaluates to an algebraic number precisely at CM points isogenous to a base CM point. This allows us to isolate the rational exceptional loci using the CM point tables from Elkies \cite{ne}.\\
 
 Finally, in Section 5, the exact exceptional values are explicitly computed by using two distinct methods tailored to the geometry of each CM point - A geometric factorization, analogous to the approach of \cite{bg}, through the study of an intermediate level-2 cover $X_{15}^*(2)$ (for $t=243$) and a level-3 cover $X_{14}^*(3)$ (for $t=75/16$). For $t = -\frac{729}{112}$, we follow an algebraic framework utilizing Yang polynomials $\Psi_p$, studied in \cite{yy2}, constructed from explicit Hecke operator matrices combined with localized numerical root isolation to eliminate spurious branches.

\begin{comment}
    
\medskip To find the value at the CM point corresponding to 243, we use a level 2 covering of the Shimura curve ($X_{15}^*(2)$) alongside an Atkin-Lehner involution. This allows us to compute the translated value of $h_4$ algebraically, utilizing a mathematical limit to resolve the multiplicity of the Hauptmodul at the base point $z_{12}$.

To evaluate the Heun function at $-729/112$, we employ (a new approach using) Yang's $\Psi_p$ modular polynomials. These polynomials relate the ratio of a modular form at two different points connected by a Hecke operator directly to the Hauptmodul. By factoring a norm-70 group element into prime-norm components ($γ_7 * ω_5 * γ_2$), we use the Hecke operators to generate the polynomials $\Psi_2$ and $\Psi_7$. This allows to chain the modular form ratios together into a product and computationally solve for the exact algebraic root.

\end{comment}

\section{Preliminaries}

\subsection{Quaternionic Modular forms on Shimura Curves}

Let $B_D$ be an indefinite quaternion algebra over $\mathbb{Q}$ of discriminant $D > 1$. A \textit{maximal order} is an order that is not properly contained in any other order. An \textit{Eichler order} is defined as the intersection of two maximal orders. See \cite[ch.2]{ab} and \cite[ch2]{ne} for a more detailed introduction to Shimura curves.

\medskip Let $B_{14}$ be the quaternion algebra with symbol $(-2,7)_ℚ$. From \cite{ne} we know that it is the $ℚ$-algebra generated by $b, g$ with relations given as $b^2=-2, g^2=7, bg=-gb$. $\OO_{14}=ℤ\left[b,\frac{(1+b+g)}{2}\right]$ is a maximal order in $B_{14}$. We choose the embedding $i: B_{14} \rightarrow M(2, \R)$ to be

$$    b \mapsto \begin{pmatrix} 0 & 1 \\ -2 & 0 \end{pmatrix} \quad g \mapsto \begin{pmatrix} \sqrt{7} & 0 \\ 0 & -\sqrt{7} \end{pmatrix}$$

%TODO include in first definition and remove above 
\medskip The complex curves $X_{14}(1) = ℍ/Γ$ and $X^*_{14}(1) = ℍ/Γ^*$ are Shimura curves of discriminant 14 over $\Q$. The group $Γ^*$ is a Fuchsian group and it has a signature $(0; 2, 2, 2, 4)$ with elliptic points $z_4, z_8, z_{56}$ and $z_{56b}$ where elliptic point $z_{\Delta}$ denotes a CM point of discriminant $-Δ$. We give the orders of the elliptic points, and the value of the Hauptmodul $u$ at these points.

We get the following data for elliptic points from  Sections 5.1 and 5.2 in \cite{ne}.\\

$$\begin{matrix} z \; & z_4 & z_8 & z_{56} & z_{56b}   \\ \hline
|{\rm Stab}(z)| \;  & 4 & 2 & 2 & 2 \\ \hline
u(z) \; & ∞ & 0 & \frac{-13 + 7\sqrt{-7}}{32} & \frac{-13 - 7\sqrt{-7}}{32} \end{matrix}$$

$$z_4=\frac{i-\sqrt{7}}{4}, \quad z_8=\frac{i}{\sqrt{2}}, \quad z_{56}=\frac{(1+\sqrt{7})(1+i\sqrt{2})}{6}, \quad z_{56b}=\frac{(1+\sqrt{7})(-1+i\sqrt{2})}{6}$$

\medskip Let $B_{15}$ be the quaternion algebra with symbol $(-3,5)_ℚ$ generated by $c, e$ with relations given as $c^2=-3, e^2=5, ce=-ec$. We choose $\OO_{15}=ℤ\left[e,\frac{1+c}{2}\right]$ as our maximal order. We choose the embedding $i: B_{15} \rightarrow M(2, \R)$ to be

$$    c \mapsto \begin{pmatrix} 0 & 1 \\ -3 & 0 \end{pmatrix} \quad e \mapsto \begin{pmatrix} \sqrt{5} & 0 \\ 0 & -\sqrt{5} \end{pmatrix}$$

\medskip The group $Γ^*_{15}$ is a Fuchsian group with signature $(0; 2, 2, 2, 6)$. It has $4$ elliptic points $z_3, z_{12}, z_{15}$ and $z_{60}$.

$$\begin{matrix} z  & z_{3} & z_{12} & z_{15} & z_{60}   \\ \hline
|{\rm Stab}(z)| \;  & 6 & 2 & 2 & 2 \\ \hline
t(z)  & ∞ & 0 & 81 & 1 \end{matrix}$$

$$z_3=\frac{i}{\sqrt{3}} \quad z_{12}=\frac{\sqrt{3}(\sqrt{15} + i)}{12} \quad z_{15}= \frac{3+\sqrt{3}i}{3(\sqrt{5}-1)} \quad z_{60}=\frac{3\sqrt{3}+2i}{\sqrt{3}(4\sqrt{5}-7)}$$

\begin{comment}

\medskip Since $B$ is split at $∞$, we fix an isomorphism $i_∞: B \otimes_{\mathbb{Q}} \mathbb{R} \xrightarrow{\sim} M_2(\mathbb{R})$. Let $\mathcal{O}_1^\times = \{ γ \in \mathcal{O}^\times : \mathrm{nrd}(γ) = 1 \}$ be the group of units in $\mathcal{O}$ with reduced norm $1$. Via the isomorphism $i_∞$, we can embed this group into $SL_2(\mathbb{R})$. Taking this image modulo its center defines an arithmetic Fuchsian group:
$$Γ(1) = \frac{i_∞(\mathcal{O}_1^\times)}{\{\pm 1\}} \subset PSL_2(\mathbb{R})$$

And the compact quotient of the upper half-plane $\mathcal{H}$

$$X(1)_\mathbb{C} = Γ(1) \backslash \mathcal{H}$$ is the associated Shimura curve. The compact Riemann surface associated to a discrete subgroup of
$SL_2(\mathbb{R})$ commensurable to some $\mathcal{O}_1^\times$ will also be called a Shimura curve. 
\end{comment}

\medskip {\rm \bf CM points of Shimura Curves:} For  an Eichler order $\OO$ of level $N$ in $B$ fix an embedding $i$ of $B$ into $M(2, \R)$. Let $\K$ be an imaginary quadratic field and $R$ an order of discriminant $d_R = f^2d_{\K}$ in $\K$. Following Eichler, we say an embedding $\phi : R \rightarrow \OO$ is optimal if $\phi(\K) \cap \OO = \phi(R)$.  
Now the action of the set $i\circ \phi(R \backslash \{0\}) \subset {\rm GL}^{+}(2, \R)$ on the upper half-plane $\mathcal{H}$ fixes precisely one point $t_{\phi}$. Such a point is called a CM-point (point with complex multiplication) of discriminant $d_R$.

\medskip The \emph{transcendental part} of a nonzero complex number $z \in ℂ$ is its image in $ℂ^*/\bar{ℚ}^*$. 
We write $a\sim b$ for two numbers $a,b∊ℂ^*$ when they have the same transcendental part or equivalently, if their ratio is algebraic. We sometimes take any representative in the equivalence class of $z$, and refer to it as the transcendental part of $z$.

Shimura proved that the period at a CM point 
with fundamental discriminant $Δ$ has a transcendental part $\sqrt{π}Ω_Δ$ which only depends on the discriminant $Δ$. 
This helps to algebraically relate the values of the periods at CM points with the same discriminant. The exceptional set of Heun function is exactly the set of CM points whose fundamental discriminant matches that for a particular CM base point. CM points with a different fundamental discriminant are also characterized as the set of points where the transcendental part is a given fixed ratio involving absolute periods.

\medskip{\rm \bf Quaternionic Modular forms:} Let $X$ be a genus-zero Shimura curve with Hauptmodul $t(z)$. Following Yang \cite{yy}, automorphic forms on $X$ can be described explicitly in terms of solutions of the Schwarzian differential equation associated to $t$. 

\begin{theorem}\label{thm:yang-basis}\cite{yy}

Assume that $X(\mathcal O)$ has genus 0 with signature $(0; e_1, \dots, e_r)$ and the corresponding elliptic points $z_i$. Let $t(z)$ be a Hauptmodul and set $a_i = t(z_i)$. For a positive even integer $k \geq 4$, let

\begin{equation}
d_k = {\rm dim} S_k(\mathcal O) = 1 - k + 
\sum_{i=1}^r \left \lfloor \frac{k}{2} \left( 1 - \frac{1}{e_i} \right) \right \rfloor 
\end{equation}

be the dimension of the space of modular forms of weight $k$ on $Γ(\mathcal O)$. Then a basis for $S_k(\mathcal O)$ is given by the following forms. %TODO use : ? 

\begin{equation}
 t(z)^j t'(z)^{\frac{k}{2}} \prod_{i=1, a_i \neq ∞}^r (t(z) - a_i) ^{- \left \lfloor {\frac{k}{2}} \left( 1 - \frac{1}{e_i} \right) \right \rfloor } \quad j=0, 1, \dots d_k-1.
\end{equation}

\end{theorem}
\medskip This approach makes explicit computation of automorphic forms and Hecke operators possible. By utilizing the Jacquet–Langlands correspondence and explicit covers between Shimura curves, Yang  also devised a method to compute Hecke operators with respect to the explicitly given basis of modular forms. The reader who is unfamiliar with the relation between automorphic forms and differential equations may like to see Proposition 5 in \cite{for} and its proof in \cite{yy2}.

\subsection{Shimura curves and Moduli} A \emph{Shimura curve} over $\mathbb{Q}$ is a natural generalization of the classical concept of a modular curve. Classical modular curves (like $X_0(N)$) are moduli spaces parametrizing elliptic curves with specific level structures. A Shimura curve $X$ has an associated indefinite quaternion algebra $B$ over $\mathbb{Q}$, with discriminant $D > 1$. It parameterizes triples $(A, \mathcal{λ}, \iota)$, where $A$ is an abelian surface, $\mathcal{λ}$ is a principal polarization, and $ι: \OO \hookrightarrow \text{End}(A)$ is a ring embedding such that the Rosati involution induced by $\mathcal{λ}$ restricts to a specific involution on $B$ (see \cite{bg1} for details).  The ring embedding introduces extra endomorphisms coming from the order $\OO$ which gives a ``quaternionic multiplication'' (QM) on  the abelian surface $A$. 

%\medskip For $z \in X$ away from the elliptic points on $X$, in the cases considered here, the associated principally polarized abelian surfaces  $A_z$ may be realized as Jacobians $\operatorname{Pic}^0(C_z)$ of  a non-singular genus 2 curves $C_z$ by Torelli's theorem. Further, Mestre's construction gives us the coefficients of the hyperelliptic equation of $C_z$ as algebraic functions of the Igusa-Clebsch invariants of the curve. Hence, the family of genus 2 curves can be defined by an affine equation of the form $y^2 = f_z(x)$, where $f_z(x)$ is a sextic polynomial whose coefficients are rational functions of the uniformizing Hauptmodul $t(z)$. 

%Locally, the Igusa-Clebsch invariants of these abelian surfaces can be expressed as quaternionic modular forms, allowing one to construct the hyperelliptic genus 2 curves $y^2 = f_z(x)$ explicitly.

\subsection{Geometry of the Heun Differential Equation} The Heun equation generalizes the hypergeometric equation to allow four regular singular points instead of three. Heun equations arise naturally in physics and in many geometric moduli problems \cite{mr}. 

We review the geometric meaning of the Picard-Fuchs equation. As mentioned, a Shimura curve parametrizes a family of principally polarized abelian surfaces $A_z$. The relative de Rham cohomology $\mathcal H^1_{dR}=R^1\pi_*\Omega^\bullet_{\mathcal A/X}$ of this family has a natural connection called the Gauss–Manin connection.
$$\nabla:\mathcal H^1_{dR}\to \mathcal H^1_{dR}\otimes\Omega_X^1$$ This allows us to define flat sections of the bundles of forms and by taking duals, we can define flat or locally constant cycles. Period functions are obtained by integrating a holomorphic form $ω∊H^{1,0}(A_z)$ on a locally constant cycle. Neither the holomorphic form nor the bundle whose fibers are $H^{1,0}(A_z)$ are locally constant in general. This means that the period, which is a solution of the Picard-Fuchs equation, is  measuring the variation of a Hodge structure (VHS) relative to flat section of the homology local system.  

\medskip As the Shimura curves of discriminants 14 and 15 are uniformized by quadrilateral arithmetic Fuchsian groups, the associated Picard-Fuchs equation has exactly four regular singularities, which after a Möbius transformation, becomes a Heun differential equation. The Möbius transform ensures that the four regular singularities are at $0, 1, a$, and $∞$. The standard form of the Heun equation for a function $w(t)$ is given as follows.

\begin{equation}
 w'' + \left(\frac{γ}{t} + \frac{δ}{t -1} + \frac{ε}{t - a} \right) w' + \frac{αβt - q}{t(t - 1)(t - a)} w = 0  \label{heun}
\end{equation}

where the parameters satisfy the Fuchsian relation $\alpha + \beta + 1 = γ + \delta + \epsilon$, and the singularity $a \neq 0, 1$. The local monodromy of the periods around the singular fibers (the elliptic points) directly determines the parameters $\alpha, \beta, γ$, and $\delta$, intrinsically linking the differential equation to the hyperbolic geometry of the Fuchsian group. The unique local power series solution near $t=0$ with a constant term of $1$ is defined as the Heun function, denoted $He(a,q;\alpha,\beta,γ,\delta;t)$.

 $$   {\rm He}(a, q; α, β, γ, δ ; t)  = \sum_{j =0}^{∞} c_j t^j  $$

We define $B_a=[1,∞)\cup \{sa:s∊[1,∞)\}$ as the union of two branch  cuts from $1$ to $∞$ and  from $a$ to $∞$ along a ray. Our domain $D_a$ for the Heun function is the complement $D_a=ℂ\backslash B_a$. In the discriminant $15$ case, we will need to define the Heun function at $243$ on $D_{81}=ℂ \backslash [1,∞)$. In that case, we will consider the limit from the upper side of the cut. Hence $He_{15}(243)=lim_{y\downarrow 0}He_{15}(243+iy) $

\subsection{Heun parameters for $X^*_{14}(1)$ and $X^*_{15}(1)$}
For the discriminant 14 case, we have the following differential equation with coefficients in $ℚ(u)$ as given in equation [71] in \cite{ne}.

\begin{equation}
u(16u^2+13u+8)w''+(24u^2+13u+4)w'+\left(\frac{3}{4}u+\frac{3}{16}\right)w=0 \label{X14eqn}
\end{equation}

This equation has 4 singular points $0, r_1, r_2, ∞$ where $r_1, r_2$ are the roots of $(16u^2+13u+8)=0$. After changing variables with $u=r_1.t$, this equation can be transformed into the standard Heun equation \eqref{heun}, so that the singular points move to $0,1,r_2/r_1, ∞$. The parameters of the Heun equation are as follows.
$$γ = δ = ε = \frac{1}{2}, \quad a=\frac{r_2}{r_1}, q=\frac{-3}{(256\cdot r_1)}, \; \; α =\frac{1}{8}, \; \; β =\frac{3}{8} \quad 
$$

$r_1 = (-13 - 7\sqrt{-7})/32 \quad r_2 = (-13 + 7\sqrt{-7})/32$. 
\medskip{}
\\[0.3cm]
When finding the rational locus of the exceptional set, we focus on a solution $F(u)$ of \eqref{X14eqn} as this is the equation which is defined over $ℚ$. $F(u)$ is defined as the scaled Heun function, $$F(u)=He(\frac{r_2}{r_1}, \frac{-3}{256.r_1}, \frac{1}{8},\frac{3}{8},\frac{1}{2},\frac{1}{2},u/r_1)$$

For $X_{15}^*(1)$, equation [75] in \cite{ne} gives us a standard Heun equation with the parameters,

$$γ = δ = ε = \frac{1}{2}, a = 81, q=\frac{1}{2}; \; \; \alpha =\frac{1}{3}, \; \; \beta =\frac{1}{6}$$

\section{Shimura curves of discriminants 14 and 15}

In this section, we construct explicit generators for the graded ring of modular forms on $X_{14}^*(1)$ and $X_{15}^*(1)$ using Yang’s result for automorphic forms on genus-zero Shimura curves. The main result in this section is that a weight $1$ modular form satisfies the Heun differential equation on an appropriate domain. This allows us to give an explicit local modular expression for the Heun function near the CM basepoint $z_8$ for $D=14$ and $z_{12}$ for $D=15$.

\begin{figure}[!h]
\centering
\begin{tikzpicture}[scale=5]

% ---------- User input: set the four vertices (xi, yi) ----------
\def\xA{-0.661437827766148}
\def\yA{0.25} %z4
\def\xB{0}
\def\yB{0.707106781186548} %z8
\def\xC{0.607625218510765}
\def\yC{0.85931182485783} %z56
\def\xD{-0.607625218510765}
\def\yD{0.5931182485783}  %z56b
\def\xAp{-0.888806} \def\yAp {0.319276} % reflection of A in CD
\def\xBp{-0.125462} \def\yBp {1.264807} % reflection of B in CD
\def\xCp{-0.4213780777} \def\yCp {0.4113177331} % reflection of A in CD
\def\xDp{0.274291885} \def\yDp {0.3879073041} % reflection of B in CD
%$C' \approx (-0.4213780777, 0.4113177331)$$D' \approx (0.2742918852, 0.3879073041)$
% -----------------------------------------------------------------

% real axis
\draw[->] (-1,0) -- (1,0) node[right] {$ℝ$};
\draw[] (0,1) -- (0,0) node [below] {$0$};
\node at (0,1.05) {$1$};
% mark the points
\coordinate (A) at (\xA,\yA);
\coordinate (B) at (\xB,\yB);
\coordinate (C) at (\xC,\yC);
\coordinate (D) at (\xD,\yD);
\coordinate (C') at (\xCp, \yCp);
\coordinate (D') at (\xDp, \yDp);

% draw the four edges (A->B->C->D->A)
\hypquad(\xA,\yA)(\xB,\yB)(\xC,\yC)(\xD,\yD){white!40};
\hypquad(\xA,\yA)(\xB,\yB)(\xDp,\yDp)(\xCp,\yCp){gray!40};
%\drawgeo{\xA}{\yA}{\xB}{\yB}
%\drawgeo{\xB}{\yB}{\xC}{\yC}
%\drawgeo{\xC}{\yC}{\xD}{\yD}
%\drawgeo{\xD}{\yD}{\xA}{\yA}
%\drawgeo{\xC}{\xAp}{\yC}{\yAp}
%\drawgeo{\xAp}{\xBp}{\yAp}{\yBp}
%\drawgeo{\xBp}{\xC}{\yBp}{\yC}

% draw and label vertices
\foreach \pt/\name in {A/$z_4$, B/$z_{8}$, D/$z_{56b}$, C'/$$, D'/$$}{
  \fill (\pt) circle (0.4pt);
  \node[below left,font=\large] at (\pt) {\name};
}

\fill (C) circle (0.4pt);
\node[below right,font=\large] at (C) {$z_{56}$};
\end{tikzpicture}
\caption{Fundamental domain for $X^*_{14}(1)$}
\end{figure}
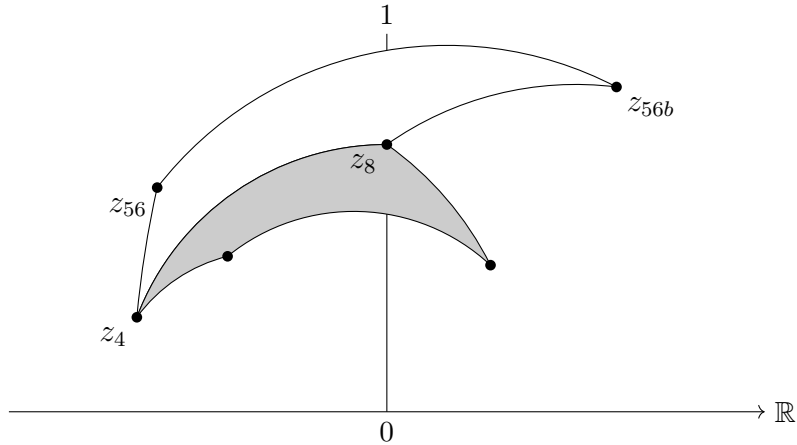

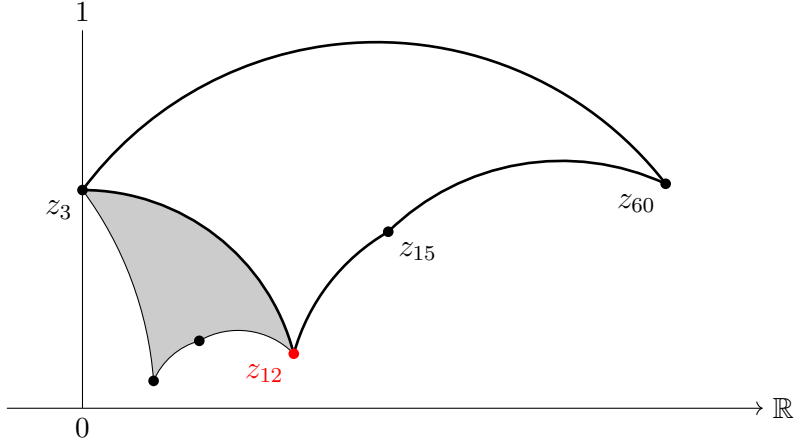
\begin{figure}[!h]
\begin{tikzpicture}[scale=5, xshift=0cm]

% ---------- User input: set the four vertices (xi, yi) ----------
\def\xA{0}
\def\yA{0.577350269189626} %z3
\def\xB{0.559016994374947}
\def\yB{0.144337567297406} %z12
\def\xC{0.809016994374947}
\def\yC{0.467086179481358} %z15
\def\xD{1.54299405580637}
\def\yD{0.593898688985200}  %z60
\def\xCp{0.3090169943} \def\yCp{0.1784110448} 
\def\xDp{0.1881553461} \def\yDp{0.0724210264} 
\def\xz{0.6716099265} \def \yz{0.0948277073}

%\def\xz48{0} \def \yz48{- 1.51152262815234}

% -----------------------------------------------------------------

% real axis
\draw[->] (-0.2,0) -- (1.8,0) node[right] {$\mathbb R$};
\draw[] (0,1) -- (0,0) node [below] {$0$};
\node at (0,1.05) {$1$};
% mark the points
\hypquad(\xA,\yA)(\xB,\yB)(\xC,\yC)(\xD,\yD){white!40}
\coordinate (A) at (\xA,\yA);
\coordinate (B) at (\xB,\yB);
\coordinate (C) at (\xC,\yC);
\coordinate (D) at (\xD,\yD);
\coordinate (Cp) at (\xCp,\yCp);
\coordinate (Dp) at (\xDp,\yDp);
\coordinate (Z) at (\xz, \yz);
\hypquad(\xA,\yA)(\xDp,\yDp)(\xCp,\yCp)(\xB,\yB){gray!40};

% draw the four edges (A->B->C->D->A)
\drawgeo{\xA}{\yA}{\xB}{\yB}
\drawgeo{\xB}{\yB}{\xC}{\yC}
\drawgeo{\xC}{\yC}{\xD}{\yD}
\drawgeo{\xD}{\yD}{\xA}{\yA}
%\drawgeo{\xA}{\yA}{\xDp}{\yDp}
%\drawgeo{\xDp}{\yDp}{\xCp}{\yCp}
%\drawgeo{\xCp}{\yCp}{\xB}{\yB}

% draw and label vertices
\foreach \pt/\name in {A/$z_3$, D/$z_{60}$, Cp/$$, Dp/$$}{
  \fill (\pt) circle (0.4pt);
  \node[below left,font=\large] at (\pt) {\name};
}

%\fill[gray!30] (A) -- (Dp) -- (Cp) -- (B) -- cycle;

\fill (C) circle (0.4pt);
\node[below right,font=\large] at (C) {$z_{15}$};

\fill[color=red] (B) circle (0.4pt);
\node[below left,font=\large, color=red] at (B) {$z_{12}$};

%\fill (Z) circle (0.4pt);
%\node[below right,font=\large, color=red] at (Z) {$z_{147}$};

\end{tikzpicture}

\caption{Fundamental domain for $X^*_{15}(1)$}
\end{figure}

\subsection{Modular Forms on $X^{*}_{14} (1)$} 

\medskip For even $k\geq 4$, the dimension formula for $M_k(X^*_{14}(1))$ in section 2.1 gives
$$   {\rm dim} M_{k}(X^{*}_{14} (1)) = 1-k+3\lf \frac{k}{4} \rf +\lf \frac{3k}{8} \rf $$

\begin{comment}
\begin{lemma} Given two non-zero modular forms $f, g$ on a Shimura curve and point $p$ such that $f(p)=0$ and $g(p)≠0$, $f$ and $g$ are algebraically independent.  
\end{lemma}
\begin{proof} Any non-trivial polynomial $P(x,y)$ with $P(f,g)=0$ will have a non-trivial homogeneous component $P_r$ (terms with the  weight r). As all the weight components in $P(f,g)$ vanish, we have $P_r(f,g)=0$. Write $P(f,g)=cg^n+f.Q(g)$ for some polynomial $Q$. As $g^n(p)≠0$ and $f.Q(g)(p)=0$, we have $c=0, P_r(f,g)=f.Q(g)=0$ which implies $Q(g)=0$. By replacing $P$ with $Q$, we get a lower weight homogeneous component which vanishes and since infinite descent is not possible on the weight, we have algebraic independence of $f$ and $g$.   
 
\end{proof}
\end{comment}

\begin{lemma}\label{lemma:af} Given a modular form $f≠0$ of positive weight on a Shimura curve and a non-constant modular function $a$, the forms $f$ and $af$ are algebraically independent.
\end{lemma}
\begin{proof} Assume there is a non-zero polynomial $P(x,y)$ with $P(af,f)=0$. Since modular forms of different weights are linearly independent, each weight component vanishes. So, we can choose $P$ to be non-zero and homogeneous of the form $\sum_0^{r}c_ix^iy^{r-i}$. Then, $0=P(af,f)=f^r\sum_0^rc_ia^i$. Since $f\not\equiv 0$, and $a$ cannot satisfy a non-zero polynomial, this is a contradiction. 
\end{proof}

\begin{theorem} The ring $M_{\bullet}(X^*_{14}(1))$ is generated by modular forms $g_4, g_8, g_{14}$ where

   $$g_4 = \frac{(Dt)^2}{t(t-1)(t-\frac{r_2}{r_1})}$$
   $$g_8 = tg_4^2=\frac{(Dt)^4}{t(t-1)^2(t-\frac{r_2}{r_1})^2}$$
   $$g_{14} =(Dt)g_4^3= \frac{(Dt)^7}{t^3(t-1)^3(t-\frac{r_2}{r_1})^3}$$

$g_4, g_8$ are algebraically independent and the relation ideal is generated by 
$$g_{14}^2=g_8^3g_4-\frac{r_1+r_2}{r_1}g_8^2g_4^3+\frac{r_2}{r_1}g_8g_4^5$$. 
    
\end{theorem}
\begin{proof}
The expressions for the modular forms follow from Theorem \ref{thm:yang-basis}. $g_4^2$ and $g_8$ are algebraically independent by Lemma \ref{lemma:af} with $a=t$. Hence $g_4, g_8$ are also algebraically independent. Since ${\rm dim} (S_{28}(X^{*}_{14} (1)))=4$, and $g_8^3g_4,g_8^2g_4^3, g_8g_4^5, g_4^7$ are linearly independent, we have the following relation for some $a,b,c,d∊ℂ$.
$$g_{14}^2=ag_8^3g_4+bg_8^2g_4^3+cg_8g_4^5+dg_4^7$$
$$g_4^6.(Dt)^2=g_4^7(at^3+bt^2+ct+d)$$
$$(Dt)^2=\frac{(Dt)^2}{t(t-1)(t-r_2/r_1)}(at^3+bt^2+ct+d)$$
$$a=1, b=-(r_1+r_2)/r_1, c=r_2/r_1, d=0$$

There cannot be any further relations between $g_{14}, g_4, g_8$ as we can use the relation on $g_{14}^2$ to rewrite any other relation so that the degree in $g_{14}$ is at most $1$.   If the new relation is not trivial, we can isolate the non-trivial homogeneous component of weight $k$.  When $4 \nmid k$, $g_{14}$ will divide all the terms as $g_4, g_8$ have weights divisible by $4$. Cancelling gives a non-trivial relation in $g_4, g_8$ which is not possible. $4|k$ implies that none of the terms involves $g_{14}$ which is similarly not possible.\\

To check that $\{g_4,g_8,g_{14}\}$  generate all forms, we note that algebraic independence implies that the set of forms $\{g_8^ig_4^{(k-8i)/4}, 0\leq i \leq \lf \frac{k}{8} \rf \}$ are linearly independent. Since there are $\lf \frac{k}{8} \rf +1=dim(S_k)$ when $k\equiv 0\ or\ 4 \pmod{8}$, they form a basis. Similarly, the set of forms $\{g_{14}g_8^ig_4^{(k-14-8i)/4}, 0\leq i \leq \lf \frac{k-14}{8} \rf \}$ forms a basis when $k\equiv 2\ or\ 6 \pmod{8}$.  

%TODO write in a more concise way
Hence, the above relation determines the ideal of relations.
\end{proof}

We note that $g_4(z_8)≠0$ as $t$ has a zero of order $2$ at $z_8$ and $(Dt)^2$ in the numerator of $g_4$ also has a zero of order $2$. 

\subsection{Modular Forms on $X^{*}_{15} (1)$} As in the case $D=14$, we can compute the dimension of the space $M_k(X^*_{15}(1))$ of weight $k$ holomorphic modular forms with $k\geq 4$.
$$   {\rm dim} M_{k}(X^{*}_{15} (1)) = 1-k+3\lf \frac{k}{4} \rf +\lf \frac{5k}{12} \rf $$

\begin{theorem} The ring $M_{\bullet}(X^*_{15}(1))$ is generated by modular forms $h_4, h_8, h_{10}, h_{12}$ where

    $$ h_4= \frac{(Dt)^2}{t(t-1)(t-81)}$$

    $$h_8=t.h_4^2=\frac{(Dt)^4}{t(t-1)^2(t-81)^2}$$

    $$ h_{10}=Dt.h_4^2=\frac{(Dt)^5}{t^2(t-1)^2(t-81)^2} $$

    $$ h_{12}=t^2.h_4^3=\frac{(Dt)^6}{t(t-1)^3(t-81)^3} $$
    
$h_4$ and $h_8$ are algebraically independent. We have the relation ideal generated by following two relations. $$h_{10}^2=h_{12}h_8-82h_8^2h_4+81h_8h_4^3 \quad \text{and} \quad  h_4h_{12}=h_8^2$$. 

\end{theorem}

\begin{proof}
  The expressions for the forms come from Theorem \ref{thm:yang-basis}. $h_4h_{12}=h_8^2$ follows immediately. $h_8=th_4^2$. So, $h_8, h_4^2$ are algebraically independent by Lemma 3.1 and hence, $h_4, h_8$ are also algebraically independent. We also have a relation $h_{10}^2=h_{12}h_{8}-82h_8^2h_4+81h_8h_4^3$ which can be checked by substituting the expressions for the forms.

 To show that there are no further relations, we first note that we can use the existing two relations to assume that any relation will  have degree at most $1$ in $h_8$ and $h_{10}$. Then, a non-trivial relation among the four generators involving $h_{10}$  gives a non-trivial relation without $h_{10}$ as $10$ is the only weight not divisible by $4$ and we can cancel $h_{10}$ from a homogeneous weight component involving it.

 Then, we can write any non-zero weighted homogeneous relation $P$ of weighted degree $r$ as $0=P(h_4,h_8,h_{12})=P(h_4,th_4^2,t^2h_4^3)=h_4^rP(1,t,t^2)$ with weights $1,2,3$ in the 3 variables.

As the degree of $P$ in the second variable is at most $1$, we have $0=P(1,t,t^2)=tQ(1,t^2)+R(1,t^2)$ where $Q, R$ are polynomials in two variables of degrees $r-2$ and $r$ with variable weights $1, 3$. Since $tQ(1,t^2)=-R(1,t^2)$ equates an odd and even  polynomial in $t$, we have, $0=Q(1,t^2)=R(1,t^2)$. If $R=Σa_ix^{r-3i}y^i, Σa_it^{2i}=0$. Hence, $a_i=0, R=0$ and similarly, $Q=0$ and no further non-trivial relation exists.

Algebraic independence and counting forms of a given weight  generated by $\{h_4, h_8, h_{10}, h_{12}\}$ show that they generate all forms.

\end{proof}

\begin{lemma} Given an extension of  integral domains $R ⊂ S$, $a,b∊S$ and $a,b^2$ algebraically independent over $R$, $a,b$ are also algebraically independent over $R$.
\end{lemma}
\begin{proof}
  We can write any polynomial relation between $a,b$ as $P(a,b^2)=b*Q(a,b^2)$ by bringing $b^{2k}$ terms to the left and $b^{2k+1}$ terms to the right. A priori, $P(x,y^2)$ need not be the same as $yQ(x,y^2)$, but after squaring and using algebraic independence of $a, b^2$ we get $P(x,y^2)^2=y^2Q(x,y^2)^2$ which implies $P(x,y^2)=\pm yQ(x,y^2)$ inside $R[x,y]$  as $R[x,y]$ is also an integral domain. But, this is not possible unless $P=Q=0$ as the LHS contains only even powers of $y$ and the RHS only odd powers, which forces the polynomial relation to be trivial.
\end{proof} 

\subsection{Modular Solutions of the Heun Equation}

 For a vertex $z$ in the quadrilateral tessellation of $\HH$, let $M(z)$ denote the open neighborhood of $z$ that consists of the interior of the union of those quadrilaterals having $z$ as one of their vertices.
 
\medskip{} We show that $g_4^{1/4}$ in $M(z_8)$ satisfies a Heun equation whose singularities correspond to the elliptic points of  $X^*_{14}(1)$. Similarly $h_4^{1/4}$ satisfies the Heun equation for $X^*_{15}(1)$

\begin{definition} Let $t∊S_0(Γ)$ be a non-constant modular function. For a given non-zero modular form $f∊S_k(Γ)$, we define $L_f(F)=0$ as the differential equation given as $$\frac{d^2F}{dt^2} + a(t)\frac{dF}{dt} + b(t)F = 0$$ 
where $$a(t) = \frac{D^2t}{(Dt)^2} - \frac{2Df}{kf(z)Dt}$$ $$b(t) = \frac{a(t)^2}{4} + \frac{da}{2dt} + Q(t)$$ with
$$Q(t) = \frac{3(D^2t)^2 - 2DtD^3t}{4(Dt)^4}$$\\

We extend this definition of $L_f$ to $t∊S_0(Γ^*), \ f∊S_k(Γ^*)$ using the inclusion  $S_k(Γ^*)⊂S_k(Γ)$.
\end{definition}

The significance of this definition arises in \cite{bg3} where it is shown that one can determine the exceptional sets of functions $h$ which satisfy the differential equation $L_f(h)=0$.

Further, \cite{z1} show that $f^{1/k}, zf^{1/k}$ locally form a basis for meromorphic solutions of $L_f$ for any branch $f^{1/k}$. 

\begin{lemma} In the open set $M(z_8)$,
 $L_{g_4}(w)=0$ is the Heun equation $$\frac{d^2w}{dt^2} + \left(\frac{1/2}{t} + \frac{1/2}{t -1} + \frac{1/2}{t - a} \right) \frac{dw}{dt} + \frac{(\frac{1}{8})(\frac{3}{8}) t - q}{t(t - 1)(t - a)} w = 0 $$
 
 $g_4^{1/4}$ is a solution to the Heun equation on $M(z_8)$.
\end{lemma}
\begin{proof}
   
   \medskip  Taking logarithmic derivatives, we get $$D(ln(g_4))=\frac{Dg_4}{g_4}=D(2ln(Dt)-ln(t)-ln(t-1)-ln(t-\frac{r_2}{r_1}))$$
    $$\frac{2Dg_4}{4g_4}=\left(\frac{D^2t}{(Dt)^2}-\frac{1}{2}\left[\frac{1}{t}+\frac{1}{t-1}+\frac{1}{(t-\frac{r_2}{r_1})}\right]\right)Dt$$
   which gives that, $$a(t)=\frac{1}{2}\left[\frac{1}{t}+\frac{1}{t-1}+\frac{1}{(t-\frac{r_2}{r_1})}\right]$$ is the coefficient of the Heun equation. That the coefficient $b(t)$ also matches is an automatic consequence as explained in proof of Proposition 4 in \cite{bg} and can also be checked by a similar computation as above.\\

   The result in \cite{z1} then implies that $g_4^{1/4}$ is locally a solution of the Heun equation. Since there are no singularities in $M(z_8)$ except $z_8$ itself where $g_4$ is finite and non-zero as both the numerator and denominator in the expression for $g_4$ have order $2$. So, it can be extended to $M(z_8)$.
    
\end{proof}

\begin{lemma} 
 $L_{h_4}(w)=0$  is the Heun equation $$\frac{d^2w}{dt^2} + \left(\frac{1/2}{t} + \frac{1/2}{t -1} + \frac{1/2}{t - 81} \right) \frac{dw}{dt} + \frac{(\frac{1}{18}) t - \frac{1}{2}}{t(t - 1)(t - 81)} w = 0 $$
 $h_4^{1/4}$ is a solution to the Heun equation on $M(z_{12})$.

\end{lemma}

The proof is the same as for $D=14$ case with $\frac{r_2}{r_1}$ replaced by $81$.

\medskip
We note that $h_4(z_{12})≠0$ as $t$ in the denominator of $h_4$ has a zero of order 2 at $z_{12}$ and $(Dt)^2$ in the numerator also has a zero of order $2$. So, we can choose a branch $h_4^{1/4}$ in $M(z_{12})$.

\medskip As a corollary of the above, we obtain the following.
 \begin{proposition}
 
  \item   In $M(z_8)$, we have
$$g_4(z) = \frac{(z_8 - \bar{z}_{8})^4 g_4(z_8)}{(z - \bar{z}_{8})^4}{\rm He}\left( \frac{r_2}{r_1}, \frac{-3}{(256.r_1)}, \frac{1}{8}, \frac{3}{8}, \frac{1}{2}, \frac{1}{2}; t(z) \right)^4$$

\item  In $M(z_{12})$, we have

     $$h_4(z) = \frac{(z_{12} - \bar{z}_{12})^4 h_4(z_{12})}{(z - \bar{z}_{12})^4}{\rm He} \left( 81, \frac{1}{2}; \frac{1}{3}, \frac{1}{6}, \frac{1}{2}, \frac{1}{2}; t \right)^4$$
 \end{proposition}
 
 The proof is the same as Proposition 5 in \cite{bg} with our  base points and modular forms.

\section{The Exceptional Set}

In this section we determine the rational locus of the exceptional set. We will see that although the exceptional set is an infinite set of CM points in $X^*_{14}(1),\ X^*_{15}(1)$ (CM points with fundamental discriminant is $-2, -3$ respectively), there are  $2$ non-trivial rational points for $D=15$ and only one for $D=14$.

\begin{definition} A meromorphic modular function $t∊S_0(Γ)$ is called algebraic when $t(z)∊\bar{ℚ} \cup \{∞\}$  at all CM points $z$. A modular form $f$ of weight $k$ is algebraic if there is an algebraic modular function $t$ such that $\frac{f}{(Dt)^{k/2}}$ is also an algebraic modular function. As $S_k(Γ^*)⊂S_k(Γ)$, we can define algebraic modular forms for $Γ^*$ by using the inclusion map. 
\end{definition}

Let $f∊S_k(Γ^*), k>0$ whose pullback to $S_k(Γ)$ is an algebraic modular form. Let $t$ be an algebraic Hauptmodul for $Γ^*$. Let $U$ be a simply connected open domain inside $ℍ$ on which $f$ is finite and non-zero. Let $z_0∊U$ be a CM point which is elliptic with respect to $Γ^*$ with $t(z_0)=0$.

\begin{proposition}\label{thm:exception-set}   Let $h(t)$ be any solution of $L_{f}$ on $t(U)$  which is holomorphic at $t = t(z_{0})$ and normalized so that $h(t(z_{0})) = 1$. Then $E(h) = t(U) \cap \{t(λ(z_{0})) \mid λ \in \OO, \text{nr}(λ) > 0\}.$
\end{proposition}
\begin{proof}
The proof of the proposition is the same as  Prop 6.2 in \cite{bg3} with small modifications.

Choose a branch $g=f(z)^{1/k}$ on $U$. Then we can write any solution $h$ of $L_f$ locally as $h(t(z))=(a+bz)f^{1/k}$. Differentiating and evaluating at $z_0$, we get 
$$\frac{Dh(t(z_0))t'(z_0)}{g(z_0)}=b+(a+bz_0)\frac{f'(z_0)}{kf(z_0)}\label{eqn:abz}$$

$t'(z_0)=0$ as $z_0$ is elliptic. So we can equate the right hand side to $0$. We define

 $$φ(z)=\frac{f'(z)}{kf(z)}+\frac{1}{z-\bar{z}}$$ 
 This is a modular form of weight $2$. (see \cite{z1} Page 51).
As $z_0$ is elliptic point, we can find a $γ∊Γ^*, γ≠id$ fixing $z_0$. This implies that
$φ(z_0)=φ(γz_0)=γ'(z_0)^{-1}φ(z_0)$. As $γ'(z_0)$ is a non-trivial root of unity, we get $φ(z_0)=0$. Hence,\\
$$\frac{f'(z_0)}{kf(z_0)}=\frac{1}{\bar{z_0}-z_0}$$.

Substituting in \ref{eqn:abz}, we get $a=-b\bar{z_0}$ and $ h(t(z))=b(z-\bar{z_0})g(z)$. \\

Assuming $t(z)$ is algebraic, Prop 5.1 in \cite{bg3} says that the $\bar{ℚ}$ vector space of periods $V_z$ on the abelian surface at $z$ is given by $\sqrt{π}g(z)(\bar{ℚ}z+\bar{ℚ})$ which implies that $h(t(z))∊\frac{b}{\sqrt{π}}V_z$. If $t(z)$ is in the exceptional set, then  $h(t(z))/h(t(z_0))$ is algebraic and  $V_z\cap V_{z_0}≠0$. Prop 5.2 in \cite{bg3} implies that there is an isogeny from $z_0$ to $z$.\\

Conversely, the same proposition states that the existence of the isogeny implies $V_z\cap V_{z_0}≠0$. But $V_{z_0}$ is one-dimensional by Prop 6.1 in \cite{bg3} as $z_0$ is  a CM point with $t(z_0)∊\bar{ℚ}$. Isogenies preserve the CM property. Then since $z$ is a CM point, $t(z)∊\bar{ℚ}$ as $t$ is algebraic and above proposition says that $V_z$ is also $1$ dimensional. Hence $V_z=V_{z_0}$ which forces $h(t(z)∊\bar{ℚ}$.

%The assumption there was that $z_0$ was a CM elliptic point with respect to $Γ$. But, what is used in the proof is that $t'(z_0)=0$ and that $f$ is a modular form with respect to a nontrivial element $γ$ which fixes $z_0$. 
\end{proof}

This means that the exceptional set is generated by CM points with the same fundamental discriminant as the basepoint $z_0$ or alternatively CM points $z$, whose corresponding abelian surfaces $A_z$  are isogenous to $A_{z_0}$.

\begin{theorem} The rational exceptional set $E \cap \mathbb{Q}$ for the Shimura curve $X_{14}^{*}(1)$ is $u∊\{0, \frac{75}{16}\}$ and similarly for $X_{15}^{*}(1)$ is $t∊\{0, 243, \frac{-729}{112} \}$.

\end{theorem}

\begin{proof} 

For our two Shimura curves, we can ensure $t$ is algebraic by choosing a canonical model. The forms $g_4, h_4$ are algebraic forms as they are a product of an algebraic function in $t$ with $(Dt)^2$.\\

By Proposition \ref{thm:exception-set}, we can determine the exceptional set by finding rational CM points whose fundamental discriminant is the same as the basepoint $z_8$ for $X_{14}^{*}(1)$ with fundamental discriminant $-8$. This list of points can be found in Table 5  in \cite{ne}. The only other point apart from $0$ is $\frac{75}{16}$ with $D=-72$. Similarly, Table 6 tells us that the points which have the same fundamental discriminant as the basepoint $z_{12}$ for $X_{15}^{*}(1)$ are 243 with $D=-48$ and $-\frac{729}{112}$ with $D=-147$.

The tables obtained in Elkies use theoretical restrictions on class groups coming from Shimura reciprocity to determine CM points $z$ with $t(z)∊ℚ$. However, the explicit value $t(z)$ is initially computed numerically. To check that the values are rational, we use the maps $η:X^*_{14}(3)→X^*_{14}(1)$, $ω_3:X^*_{14}(3)→X^*_{14}(3)$ which are defined over $ℚ$, and the images of the CM points $x=z_8$ to $x'=z_{72}$ are related as $t_3(x')=ω_3(t_3(x))=ω_3(1/5)=-1/5$. $η(-1/5)=75/16$ which proves the correctness of the value. \\

For discriminant $15$, a similar calculation works for the $t=243$ but not for $t=-729/112$. Guo and Yang in \cite {gy} Example 32,  use Borcherds forms to verify the list of CM rational points on $X^*_{15}(1)$ given in Elkies.  

\end{proof}

\section{Algebraic values of the Heun function} We adopt two different approaches to calculate the algebraic values at the exceptional points 243 and $\frac{-729}{112}$ respectively. In the first approach following \cite{bg}, we factorize the map from $ℍ→X^*(1)$ through $X^*(2)$. This factorization helps to cancel certain transcendental factors involving derivatives of $t$. We use this to calculate $He_{15}(243)$. This approach requires a preimage of the target point to be in the Atkin-Lehner orbit of preimages of the basepoint. This fails for $He_{15}(\frac{-729}{112})$ so we use Yang polynomials which allow us to calculate the modular form ratio $h_4(z_{147})/h_4({z_{12}})$, which is the key ingredient in the expression for the Heun function. The second approach is more general and also works for $t=243$, but we still use the factorization approach as the calculation follows from geometric reasoning involving Atkin-Lehner maps on $X^*(2)$, which might be of independent interest. 

\begin{theoremb}\label{thm:heun_factorization}
\hspace{0.9em}   (i)$$    {\rm He}_{15}(243) = \frac{\sqrt{3}}{2} + \frac{i}{4}$$

(ii)$$ {\rm F}_{14}(\frac{75}{16}) = \frac{\sqrt{7}}{3}$$

\end{theoremb}
\begin{proof} (i)
\medskip Let $z_{48}=γz_{12}$ where $γ=3-2c+(3e+ce)/2$. Here $γ$ represents the Atkin-Lehner involution $ω_2$ as $γ$ normalizes $\OO_2=ℤ<1,1+c,(1+c+2e)/2, (1+c+e+ce)/2>$ which is an Eichler order of level $2$. $γ$ has norm $6$, represents $ω_2.ω_3$ on $X_{15}(2)$ and hence $ω_2$ on $X_{15}^*(2)$. $t(z_{48})=243$ by Table 6 in \cite{ne}\\

$z_{48}$ indeed has CM discriminant $-48$ as can be checked by computing the discriminant of the quadratic order which fixes $z_{48}$, $ℚ(γe_{12}γ^{-1})⋂\OO$  where $e_{12}=4c-3e$ is an element in $B_{15}$ which fixes $z_{12}$. To compute the value of the Heun function at $243$, we take a coordinate $t_2$ on $X^*_{15}(2)$ and use the  covering map $η:X^*_{15}(2)→X^*_{15}(1)$ as given in \cite[Lemma 13.2]{tu}.

$$η(t_2)=\frac{27(1-t_2)(1-3t_2)^2}{(2(1+t_2)^3)}$$

    \medskip Further, the Atkin-Lehner involution is given by $ω_2(t_2)=-t_2$. Both maps are as computed in \cite{tu} Lemma 13.2.  \\
    
    $t_2(z_{12})$ could potentially be $1/3$ or $1$  as $t(z_{12})=η(t_2(z_{12}))=0$. $η$ has multiplicities $2,\ 1$ at $\frac{1}{3},\ 1$ respectively. Since $e_{12}$ does not normalize $\OO_2$, the stabilizer group for level 2 is strictly contained in the full stabilizer group which implies that the multiplicity of $η$ at $t_2(z_{12})$ is 2. Hence, $t_2(z_{12})=\frac{1}{3}$.

    $$\begin{aligned} {\rm He}_{15}^{4}(243) &= {\rm He}_{15}^{4}(t(γ z_{12})) \\ &= \frac{(γ z_{12} - \overline{z}_{12})^{4}}{(z_{12} - \overline{z}_{12})^{4}} \frac{h_{4}(γ z_{12})}{h_{4}(z_{12})}. \end{aligned}$$

Because $t$ has a zero of order 2 at $z_{12}$, we compute $h_4(z_{12})$ by taking the limit.

$$\begin{aligned} h_{4}(z_{12}) &= \lim_{z \to z_{12}} \frac{(Dt)^{2}}{81t} \\ &= \frac{2}{81} D^{2}(t)(z_{12}) \\ &= \frac{2}{81} η''(\frac{1}{3})(Dt_2(z_{12}))^{2} \\ &= \frac{27}{16} (Dt_2(z_{12}))^{2}. \end{aligned}$$

Next, we calculate the translated form $h_4(γ z_{12})$ using the Atkin-Lehner involution properties, where $ω_2'(1/3) = -1$ and $η'(-1/3) = -\frac{8019}{4}$:

$$\begin{aligned} h_{4}(γ z_{12}) &= \frac{Dt(γ z_{12})^{2}}{243(243-1)(243-81)} \\ &= \frac{D(η t_2γ)_{z_{12}}^{2}}{Dγ(z_{12})^{2} \cdot 9526572} \\ &= \frac{D(η ω_2t_2)(z_{12})^{2}}{Dγ(z_{12})^{2} \cdot 9526572} \\ &= \frac{27}{64} \frac{Dt_{2}(z_{12})^{2}}{Dγ(z_{12})^{2}}. \end{aligned}$$

   Dividing these two results and substituting them back into our expression for the Heun function yields:

$$\begin{aligned} {\rm He}_{15}^{4}(243) &= \frac{(γ z_{12} - \overline{z}_{12})^{4}}{(z_{12} - \overline{z}_{12})^{4}} \frac{1}{4Dγ_{z_{12}}^{2}} \\ &= \left(\frac{\sqrt{3}}{2} + \frac{i}{4}\right)^{4}. \end{aligned}$$

   \medskip We can use a numerical calculation of the first few digits of ${\rm He}_{15}(243)$ to rule out $3$ of the 4-th roots to get ${\rm He}_{15}(243)=\frac{\sqrt{3}}{2} + \frac{i}{4}$.

(ii) The proof is the same as part (i). $γ=-7b-bg$ normalizes the level 3 Eichler order $\OO_3=ℤ<1,3(1+b+g)/2,b(1+g)/2,(1+3b+g)/2>$. It represents $ω_3$.$z_{72}=γz_{8}$ has CM discriminant $-72$ as $b$ fixes $z_{8}$ and $ℚ(γbγ^{-1})⋂\OO$ has discriminant $-72.$\\

We denote the Hauptmodul on $X_{14}^*(3)$ as $t_3$. The covering map $η:X_{14}^*(3)→X_{14}^*(1)$ is given in \cite{tu} Lemma 12.2 as $$η(t_3)=\frac{4(1+2t_3^2)(1-5t_3)^2}{9(1+t_3)^4},\  ω_3(t_3)=-t_3$$
and the base point is $z_{8}=i/\sqrt{2}$ with $t_3(z_8)=1/5$. $η(1/5)=0$ as $b$ does not normalize $\OO_3$ and $1/5$ is the unique point with multiplicity $2$ over $t=0$.
\end{proof}

\subsection{Modular Polynomial $Φ_2$ for $X^*_{15}(1)$}

  \begin{lemma} The modular polynomial $Φ_2$ for $X^*_{15}(1)$ is %$$Φ(x,t)=4x^3t^3 - 59049x^2t^2 + 511758x^2t + 511758xt^2 - 531441x^2 + 131797368xt - 531441t^2 + 258280326x + 258280326t - 31381059609$$

  \begin{align*}
\Phi_2(x, t) = & \, (4x^3t^3 - 59049x^2t^2 + 511758x^2t + 511758xt^2 - 531441x^2 \\
                  & + 131797368xt - 531441t^2 + 258280326x + 258280326t - 31381059609)/4t^3.
\end{align*}
  \end{lemma}
  \begin{proof} 
    We have the covering $$η:X^*_{15}(2)→X^*_{15}(1), \quad η(t_2)=\frac{27(1-t_2)(1-3t_2)^2}{2(1+t_2)^3}$$ and $ω_2(t_2)=-t_2$. As $ω_2$ is represented by an element of norm $2$, we have $η(ω_2(t_2(z)))=t(γ_jz)$  for some $γ_j$ of norm $2$. As $η(t_2(z))=t(z)$, we have a polynomial relation between $t(z)$ and $t_2(z)$, $27(1-t_2)(1-3t_2)^2-2t(1+t_2)^3=0$. There is another such relation between $t(γ_jz)=η(-t_2(z))$ and $t_2(z)$,  $27(1+t_2)(1+3t_2)^2-2t(γ_jz)(1-t_2)^3=0$. Taking resultants to eliminate $t_2$, we get a degree 3 polynomial relation between $t(z), t(γ_jz)$ which is exactly our modular polynomial with $x=t(γ_jw)$ after dividing by $4t^3$ to make it monic as per convention.
\end{proof}

\subsection{Yang Polynomial $Ψ_p$} 
 Computing the Heun function on the exceptional set requires calculating  the ratio $h_4(γz)/h_4(z)$, of the modular form $h_4$, at two CM points related by an element $γ$ in the quaternion algebra of norm $p$. The Yang polynomial $Ψ_p$,  for a prime $p$ which does not divide the discriminant, is defined in \cite{yy1} along with the process to compute it.  The polynomial $Ψ_p$ relates the above modular form ratio (transformed by a known automorphy factor) to the  Hauptmodul $t$. So the  ratio can be computed as an implicit function in terms of $t$.  The original motivation for defining $Ψ_p$ was to calculate the modular polynomial $Φ_p$. The computation of  $Ψ_p$ starts with a low degree modular polynomial $Φ_2$, and proceeds as  $$Φ_2 → T_2 → T_p → Ψ_p$$. Here $T_p$ is the Hecke operator of index p and we are interested in computing its matrix representation with respect to the Yang basis of modular forms.
 
 More explicitly, let $γ_j=\begin{bmatrix} a_j & b_j \\ c_j & d_j \end{bmatrix}$ run through representatives of the cosets used to define the Hecke operator $T_p$ action of a modular form $F$. \\Let  $F_j(w)=\frac{det(γ_j)^{k/2}}{(c_jw+d_j)^k}F(γ_jw)=(γ_j')^{\frac{k}{2}}F(γ_jw)$ then, $$T_p(F)=p^{k/2-1}\sum\limits_{j=0}^{p}F_j(w)$$
 For us, $F=h_4$.
 
 %todo clarify representatives All the $γ_j$ have norm $p$ 

\begin{definition} For a genus zero Shimura curve $X$, we let $z∊X, w∊ℍ$ and let $t:ℍ→X$  a Hauptmodul.

\medskip{} Given $F$, a modular form on $X$ of weight $k$, Yang defines the bivariate polynomial $$Ψ_p(z, t(w))=\prod\limits_{j=0}^p{\left( z-\frac{F_j(w)}{F(w)}\right)}$$  
\end{definition}

  %and . and the Hecke operator $T_p$ can be defined as $$T_p(F)=p^{k/2-1}\sum\limits_{j=0}^{p}F_j(w)$$ which implies $$T_{p}(F^r)=p^{kr/2-1}\sum\limits_{j=0}^{p}F_j(w)^r$$ 

\begin{remark}  Defining $R_{γ_j}(w):=F_j(w)/F(w)$ which is the $j$th root of $Ψ_p$, we can express the modular form ratio $F(γ_jw)/F(w)=R_{γ_j}(w)/(γ'(w))^{k/2}$. When $F$ is a weight $k$ form with respect to a group containing $γ_j$, $R_{γ_j}(w)=1$.
\end{remark}
\medskip 
\subsection{Computation of $\mathbf{Ψ_p}$}  
\begin{itemize}
\item $\mathbf{Φ_2→T_2}$ The modular polynomial $Φ_2$ is computed in Lemma 5.2. 

To calculate the matrix $T_2$ with respect to the basis, first note that each of the basis forms is expressed in terms of $t(w)$ and $t'(w)$. A typical basis form is of the form $$F=\frac{t^j(t')^{k/2}}{t^a(t-1)^b(t-81)^c}$$  However, $T_2(F)$ has terms $F(γ_jw)$ which are expressed in terms of $t(γ_jw), t'(γ_j(w))$. If we can express $t(γ_jw)$ as a function of $t$, we can write $T_2(F)$ purely in terms of $t$ and $t'$, which can then be decomposed into basis forms.

 \medskip{} The modular polynomial $$Φ_2(x,t(w))=\prod_{j=0}^{2}(x-t(γ_jw))$$ is exactly what is needed to write $t(γ_jw)$ in terms of $t$. This is because $Φ_2(t(γ_jw),t(w))=0$, which allows $t(γ_jw)$ to be computable as an implicit function series in  $t$.
\medskip{}
\medskip{}
\item $\mathbf{T_2→T_p}$ The two Hecke operators $T_2, T_p$ commute with each other. Assuming $T_2, T_p$ don't have repeated eigenvalues, they have exactly the same set of eigenvectors, but with possibly different eigenvalues.  Given the matrix of $T_2$, calculating the matrix for $T_p$ reduces to calculating its eigenvalues. 

\medskip{} By Jacquet-Langlands, the eigenvalues of $T_p$ on the Shimura curve $X^0_{N}(1)$ are the same as the eigenvalues for newforms of the classical modular group $Γ_0(N)$. To get the eigenvalues for the Shimura curve $X^*_{N}(1)=X^0_N(1)/<ω_p,\space p|D>$, we need the forms fixed by all the Atkin-Lehner maps $ω_p, p|D$. The correspondence flips the action of $ω_p$, see   \cite[Proposition 4]{yy1}. So, the forms on $X^*_{N}(1)$ will correspond to simultaneous $-1$ eigenspace of the $ω'_p$ which are the Atkin-Lehner involutions on the newforms on $Γ_0(N)$ . These newform eigenvalues along with the Atkin-Lehner actions are available via PARI-GP and also, online databases like LMFDB.
\medskip{}
\medskip{}
\item $\mathbf{T_p→Ψ_p}$  The coefficients of $Ψ_p$, as a polynomial in $z$, are elementary symmetric polynomials of $F_j(w)/F(w)$. On the other hand, if we know the explicit matrices for $T_{p}$ for the weights $kr$ and the decomposition of $F$ in terms of basis forms,  we can compute the power sums $\sum\limits_{j=0}^{p} F_j(w)^{r}$ in terms of the basis forms. Dividing by $F^r$, we also get the weight 0 form, $$\sum\limits_{j=0}^{p} \left(\frac{F_j(w)}{F(w)}\right)^{r}$$. 

 \medskip We can use the Girard-Newton identities to convert the above power sums to elementary symmetric polynomials, and thus get the coefficients of $Ψ_p$. 
\end{itemize}

Using the above process, we get the following expressions for $Ψ_2, Ψ_7$

\begin{flalign*}
&Ψ_2(z,t)=z^3 - (1/2)z^2 - (1/27)zt - (1/2916)t^2 + (1/16)z, \\
&Ψ_7(z,t)
 = z^8 + (24/7)z^7 + (64/1323)z^6t
   - (512/1750329)z^5t^2 &&\\
&\quad + (1620/343)z^6 + (8576/64827)z^5t
   + (512/1361367)z^4t^2 &&\\
&\quad + (10305536/794280046581)z^3t^3
   + (65536/3063651608241)z^2t^4 &&\\
&\quad + (7864/2401)z^5 + (19904/151263)z^4t
   - (14130176/29417779503)z^3t^2 &&\\
&\quad + (94748672/5559960326067)z^2t^3 &&\\
&\quad + (10059907072/360435548057945409)zt^4
   + (138774/117649)z^4 &&\\
&\quad + (86753024/1089547389)z^3t
   - (48782336/68641485507)z^2t^2 &&\\
&\quad + (176973824000/13349464742886867)zt^3
   + (65536/3063651608241)t^4 &&\\
&\quad + (224677312/7626831723)z^2t
   - (455722998272/494424620106921)zt^2 &&\\
&\quad + (16384/29417779503)t^3
   + (7769448/40353607)z^3 &&\\
&\quad + (1127020/282475249)z^2
   + (3126642816/678223072849)zt &&\\
&\quad + (1536/282475249)t^2
   - (1024746552/678223072849)z &&\\
&\quad + (46656/1977326743)t
   + 531441/13841287201. &&
\end{flalign*}

\subsection{Evaluating the Heun function at $\left(-\frac{729}{112}\right)$}

 We evaluate the Heun function at the rational exceptional point $-\frac{729}{112}$ by calculating the ratio of the modular form at two different CM points. 

\begin{theorem} $${\rm He}_{15}\left(-\frac{729}{112}\right)=\left(\frac{2^2 \cdot 3^3 \cdot 5^3}{7^5}\right)^{\frac{1}{6}}  $$
\end{theorem}

% Ψ def, computing Ψ using Hecke operators, triple ratio
\begin{proof}

 Just as in the $243$ case, we have 
 
 \begin{align*} {\rm He}_{15}^4\left(-\frac{729}{112}\right) = {\rm He}_{15}^{4}(t(γ z_{12}))  &= \frac{(γ z_{12} - \overline{z}_{12})^{4}}{(z_{12} - \overline{z}_{12})^{4}} \frac{h_{4}(γ z_{12})}{h_{4}(z_{12})}  \end{align*} 
 
 with $γ=25c-19e$.  This is the element in the quaternion algebra which moves $z_{12}$ to $z_{147}$, a point in $M(z_{12})$ with CM discriminant $-147$. Hence, computing ${\rm He}_{15}$ at $-\frac{729}{112}$ reduces to computing $\frac{h_4(γz_{12})}{h_4(z_{12})}$. The polynomials $Ψ_p$ for $p$ prime allow us to compute this ratio when the norm of $γ$ is $p$. So, we   factorize $γ$ which has norm $70$, compute the ratios of the individual factors and then multiply them to obtain the ratio.

  \medskip  Since $Ψ_p$ is defined only for prime $p$, we factorize $γ=γ_7ω_5γ_2$ where $γ_7=3c - 2e$ has norm 7, $ω_5=5-e+ce$ has norm 5 and $γ_2=2-c+e$ has norm $2$. 

$$ z_{12}\xrightarrow{γ_2}z'_{3}\xrightarrow{ω_5}z_{3}\xrightarrow{γ_7}z_{147}$$
 $$t(z_{12})=0, t(z'_{3})=t(z_{3})=∞, t(z_{147})=-729/112$$

$$\begin{aligned} \frac{h_{4}(γ z)}{h_{4}(z)} &= \frac{h_{4}(γ_{2}z)}{h_{4}(z)} \cdot \frac{h_{4}(ω_5γ_{2}z)}{h_{4}(γ_{2}z)} \cdot \frac{h_{4}(γ_{7}ω_5γ_{2}z)}{h_{4}(ω_5γ_{2}z)}=\frac{R_{γ_2}(z)R_{ω_5}(γ_2z)R_{γ_7}(ω_5γ_2z)}{(γ'(z))^2}\end{aligned}$$

$$=R_{γ_2}(z)R_{γ_7}(ω_5γ_2z)/(γ'(z))^2$$

$R_{ω_5}(γ_2z)=1$ as $h_4$ is modular with respect to $ω_5$. We have to take the limit as $z→z_{12}$ instead of directly evaluating at $z_{12}$ as the factors go to $0$ and $∞$, although the product is a finite number.

We can get $R_{γ_2}(z_{12})$ and $R_{γ_7}(ω_5γ_2z_{12})$ by using the two explicit polynomial relations $Ψ_2(z,t)=0, Ψ_7(z,t)=0$ and solving for $z$ in terms of $t$, locally near $t=0$. There will be multiple roots, but numerical calculations for the Heun function give us a sufficiently small error range to rule out extraneous roots.

%The first and third factors in the product are $0$ and $∞$. This can be seen by looking at orders of elliptic points (which is one greater than the zero order of $Dt$) and the formula for $h_4$, we get that $h$ has a zero of order $4$ at $z_3, z'_{3}$, and a non-zero value at $z_{12}$ and $z_{147}$. \\

The implicit function branches which are compatible with numerically determined range for the Heun function are as follows.\\
$$z_{Ψ_2}(t)\sim 4/729t^2+64/19683t^3 $$ %+O(t^4)which solves $Ψ_2$,\\
$$z_{Ψ_7}(t) \sim \frac{4}{441}\ 7^{2/3}4^{1/3}e^{2πi/3}t^{2/3}$$ %which solves $Ψ_7$. \\
Note that this implicit function $z_7$ has to be applied on $ω_5t(γ_2z)$.   To find $t(γ_2z)$, we use the modular polynomial relation $Φ_2(t(γ_2z),t(z))=0$ and take the implicit function $$z_{Φ_2}\sim\frac{531441}{4x^3} - \frac{255879}{2x^2}$$. 

$ω_5$ acts trivially on $X^*_{15}(1)$. Finally, we take the limit as $t→0$ to get our result,     
  
  $$He_{15}(-\frac{729}{112})=(\frac{2^2.3^3.5^3}{7^5})^{\frac{1}{6}}  $$.

\end{proof}

\end{document}